\documentclass[final, 1]{elsarticle}

\usepackage{latexsym}
\usepackage{float}
\usepackage{amsfonts}
\usepackage{amssymb}
\usepackage{makeidx}     
\usepackage{graphicx}    

\usepackage{comment}
\usepackage{amstext,amsmath}
\usepackage{amsthm}
\usepackage{enumerate}
\usepackage{makeidx}
\usepackage[all]{xy}
\usepackage[vcentermath]{youngtab}

 \usepackage{bigstrut}
 \usepackage{enumerate}
 \usepackage{amsthm}
 \usepackage[english]{babel}
 \usepackage[all]{xy}

\newtheorem{theorem}{Theorem}[section]
\newtheorem{thm}[theorem]{Theorem}

\newtheorem{example}[theorem]{Example}
\newtheorem{lem}[theorem]{Lemma}
\newtheorem{cor}[theorem]{Corollary}

\newtheorem{prop}[theorem]{Proposition}

\newtheorem{conj}[theorem]{Conjecture}

\newcommand\stackplus[1]{\makebox[0ex][l]{$+$} \raisebox{-.75ex}{\makebox[2ex]{$_{#1}$}}}

\newcommand{\SG}{\mathcal{S}(G)}

\newcommand{\SGp}{\mathcal{S}_+(G)}
\newcommand{\Mp}{\operatorname{M}_+}
\newcommand{\nul}{\operatorname{null}}
\newcommand{\M}{\operatorname{M}}
\DeclareMathOperator{\T}{T}

\DeclareMathOperator{\TT}{\mathcal{T}}
\DeclareMathOperator{\PP}{\mathcal{P}}

\makeindex             

\begin{document}
\title{On the Relationships between Zero Forcing Numbers and Certain Graph Coverings}


 \author{Fatemeh Alinaghipour Taklimi}
 \ead{alinaghf@uregina.ca}

 \author{Shaun Fallat\corref{cor2}\fnref{fn1}}
 \ead{shaun.fallat@uregina.ca}
 \fntext[fn1]{Research supported by an  NSERC Discovery Research Grant.}
 
 \author{Karen Meagher\fnref{fn2}}
 \ead{karen.meagher@uregina.ca}
 \cortext[cor2]{Corresponding author}
 \fntext[fn2]{Research supported by an NSERC Discovery Research Grant.}

 \address{Department of Mathematics and Statistics,\\
 University of Regina, 3737 Wascana Parkway, S4S 0A4 Regina SK, Canada}

 \begin{abstract}
   The zero forcing number and the positive zero forcing number of a
   graph are two graph parameters that arise from two types of graph
   colourings. The zero forcing number is an upper bound on the minimum number
   of induced paths in the graph, while the positive zero forcing
   number is an upper bound on the minimum number of induced trees in the
   graph. We show that for a block-cycle graph the zero forcing
   number equals the path cover number. We also give a purely graph
   theoretical proof that the positive zero forcing number of any
   outerplanar graphs equals the tree cover number of the graph. These
   ideas are then extended to the setting of $k$-trees, where the
   relationship between the positive zero forcing number and the tree
   cover number becomes more complex.
    \end{abstract}


\maketitle

\section{Introduction}

The zero forcing number of a graph was introduced in \cite{MR2388646}
and the related terminology was extended in \cite{MR2645093}. Since
then this parameter has been considered by a wealth of researchers, see, for example,
\cite{estrella, param, EHHLR, Yeh, MR2917419} for additional sources
on this topic.  Independently, physicists have studied this parameter,
referring to it as the graph infection number, in conjunction with
control of quantum systems \cite{burgarth2007full,
  burgarth2009indirect, ZFQC, Sev}. It also arises in computer science
in the context of fast-mixed searching~\cite{Yang}.

In general, when determining the zero forcing number of a graph we
start with a set of initial vertices of the graph (which we say are
coloured black, while all other vertices are white). Then, using a
particular colour change rule applied to these vertices, we change the
colour of white vertices in the graph to black.  The repeated
application of this colour change rule partitions the graph into
disjoint induced paths and each of the initial vertices is an end
point of one of these paths. The challenge is to determine the
smallest set of initial vertices so that by repeatedly applying the
colour change rule will change the colour of every white vertex of the
graph to black. Recently a refinement of the colour change rule was introduced
(called the positive zero forcing colour change rule) using this rule,
the positive semi-definite zero forcing number was defined (see, for
example, \cite{MR2645093, ekstrand2011positive, ekstrand2011note}).
When the positive zero forcing colour change rule is applied to a set
of initial vertices of a graph, the vertices are partitioned into
disjoint induced trees, rather than paths. These parameters are both
remarkable since they are graph parameters that provide an upper bound
on the algebraic parameters of maximum nullity of both symmetric and positive
semi-definite matrices associated with a graph (see \cite{MR2645093, param}). 
That is, for a given graph $G=(V,E)$, we define
\[ \SG = \{ A=[a_{ij}] : A=A^{T}, \; {\rm for} \; i \neq j, a_{ij} \neq 0 \leftrightarrow \{i,j\} \in E\},\]
and let $\SGp$ denote the subset of  positive semi-definite matrices in $\SG$. 
The {\em maximum nullity} of $G$ is  
$\M(G) = \max\{\nul (B) : B \in \SG \}$ and define
$ \Mp(G) = \max \{  \nul (B) : B \in \SGp \}, $ 
is called the {\em maximum positive semi-definite nullity of $G$}. (Here 
$\nul (B)$ denotes the nullity of the matrix $B$.)

In the next section, we define the first colour change
rule, along with stating the definition of a zero forcing set and the zero
forcing number of a graph. The relationship between these sets and
induced paths in the graph will become clear with these definitions.  In
Sections~\ref{Graphs with Z(G)=P(G)} and \ref{doublepaths} we prove
the equality between the zero forcing number and the path cover number
for the following three families of graphs: block-cycle graphs,
double paths, and graphs that we call series of double paths.  In
Section 5 we define the positive semi-definite colour
change rule along with the positive semi-definite forcing number and
forcing trees.  The positive semi-definite forcing number for
graphs that are formed by the graph operation called the vertex
sum are considered in Section~\ref{vertex_sum_of_two_graphs}.  In
Sections~\ref{doubletrees} and \ref{outerplanar} we establish results
similar to those in Sections~\ref{Graphs with Z(G)=P(G)} and
\ref{doublepaths} for the positive semi-definite forcing number.
Specifically, we show that the parameters tree cover number and
positive zero forcing number agree on double trees and outerplanar
graphs. Finally, in Section~\ref{k_trees} we give some families of
$k$-trees in which we can track both the tree cover number and the positive
zero forcing number.

\section{Zero forcing sets}
\label{zeroforcingset}

Let $G$ be a graph in which every vertex is initially coloured either
black or white. If $u$ is a black vertex of $G$ and $u$ has exactly
one white neighbour, say $v$, then we change the colour of $v$ to
black; this rule is called the \textsl{colour change rule}. In this
case we say ``$u$ forces $v$'' which is denoted by {$u\rightarrow
  v$}. The procedure of colouring a graph using the colour rule is
called a \textsl{zero forcing process} or simply a {\em forcing
process}. Given an initial colouring of $G$, in which a set of the
vertices is black and all other vertices are white, the
\textsl{derived set} is the set of all black vertices resulting from
repeatedly applying the colour-change rule until no more changes are
possible. If the derived set for a given initial subset of black
vertices is the entire vertex set of the graph, then the set of
initial black vertices is called a \textsl{zero forcing set}. The
\textsl{zero forcing number} of a graph $G$ is the size of the
smallest zero forcing set of $G$; it is denoted by $Z(G)$. A zero
forcing process is called \textsl{minimal} if the initial set of black
vertices is a zero forcing set of the smallest possible size.

For any non-empty graph $G$
\[
1\leq Z(G) \leq |V(G)|-1.
\]
The lower bound holds with equality if and only if $G$ is a path and
the upper bound holds with equality if and only if $G$ is a complete
graph. In fact, the parameter $Z(G)$ is also interesting since $\M(G) \leq Z(G)$
\cite{MR2645093}.

If $Z$ is a zero forcing set of a graph $G$, then it is possible to
produce a list of the forces in the order in which they are performed
in the zero forcing process. This list is called the
\textsl{chronological list of forces}.  A \textsl{forcing chain} is a
sequence of vertices $(v_1,v_2,\ldots,v_k)$ such that $v_i\rightarrow
v_{i+1}$, for $i=1,\ldots,k-1$ in the forcing process. The forcing
chains are not unique for a zero forcing set, as different zero forcing
processes can produce distinct sets of forcing chains.

In every step of a forcing process, each vertex can force at most one
other vertex; conversely every vertex not in the zero forcing set is
forced by exactly one vertex.  Thus the maximal forcing chains
partition the vertices of a graph into disjoint induced paths. The
number of these paths is equal to the size of the zero forcing set and the elements of
the zero forcing set are the initial vertices of the forcing chains and hence end-points of
these paths (see \cite[Proposition 2.10]{MR2645093} for more details).

A \textsl{path covering} of a graph is a family of induced disjoint paths in
the graph that cover (or include) all vertices of the graph. The minimum number of
such paths that cover the vertices of a graph $G$ is the \textsl{path cover
number of $G$} and is denoted by $P(G)$.  Since the forcing chains form
a set of covering paths we have the following basic result. 

\begin{prop}\label{P(G)<=Z(G)}
Let $G$ be a graph, then $P(G)\leq Z(G)$.\qed
\end{prop}

This inequality can be strict, in fact, complete graphs are
examples for which the difference between these two parameters can be
arbitrarily large, since if $n >3$ we have
\[
P(K_n) = \left \lceil \frac{n}{2} \right\rceil < n - 1 = Z(K_n).
\]
Conversely, there are many examples where this inequality holds with
equality, for example a path on $n$ vertices. In the next section we
consider families of graphs for which the zero forcing number equals
the path cover number.


\section{Block-cycle graphs}\label{Graphs with Z(G)=P(G)}

The most famous family of graphs for which the path cover number
agrees the zero forcing number is trees (see \cite[Proposition
4.2]{MR2388646}).  In this section, we establish this equality for the
block-cycle graphs.  We refer the readers to \cite{BFH} where some initial work 
comparing M and $P$ appeared
for block-cycle graphs.

A graph is called \textsl{non-separable} if it is connected and has no
cut-vertices.  A \textsl{block} of a graph is a maximal non-separable induced
subgraph. A \textsl{block-cycle} graph is a graph in which every block is
either an edge or a cycle. A block-cycle
graph with only one cycle is a \textsl{unicyclic} graph. Further, according to
the definition, the only block-cycle graphs with no cut vertex is
either a cycle or an edge.

It is not hard to see that in a block-cycle graph each pair of cycles
can intersect in at most one vertex, since otherwise there will be a
block in the graph which is neither a cycle nor an edge.  Thus two
blocks are said to be \textsl{adjacent} if they have exactly one
vertex in common.  A block in a block-cycle graph is a \textsl{pendant
  block} if it shares only one of its vertices with other blocks
of the graph. The next result demonstrates that just as a tree must have at
least two pendant vertices, a block-cycle graph must have at least two
pendant blocks.

\begin{lem}\label{pendant block}
 Any block-cycle graph has at least two pendant blocks.
\end{lem}
\begin{proof}
  Let $G$ be a block-cycle graph with exactly $n$ blocks. To prove
  this theorem, we will construct a graph with $n$ vertices that is
  a minor of $G$. We will show that the end-points of the longest
  induced path in $G'$ are associated to the pendant blocks in the
  original graph.

  Let $B_1,\ldots, B_n$ be the blocks in $G$. Note that $B_i$ is
  either an edge or a cycle, for any $1\leq i\leq N$.

  The vertices of $G'$ are $\{v_1,v_2,\dots,v_n\}$ and the vertex
  $v_i$ is associated to the block $B_i$. Vertices $v_i$ and $v_j$ are
  adjacent in $G'$ if and only if the blocks $B_i$ and $B_j$ share a
  vertex.

  Let $P=\{u_1,u_2,\ldots, u_k\}$ be the longest induced path in $G'$
  and assume that each $u_i$ corresponds to the block $B_i$. We will
  show that both $B_1$ and $B_k$ are pendant blocks.

  If $B_1$ is not a pendant block, then there is another block, $B$,
  that shares a vertex with $B_1$ different from the vertex that $B_1$
  shares with $B_2$. Since $P$ is a longest induced path, the block
  $B$ must also share a vertex with some $B_i$ with $i
  =2,\dots,k$. However, in this case the blocks $B, B_1,B_2, \dots, B_i$ form a
  non-separable subgraph that is neither a cycle nor an edge.
\end{proof}

The following lemma is straightforward to prove.
\begin{lem}\label{P(G-B), G is a block-cycle}
If $B$ is a pendant block in a block-cycle graph, then 
\[
P(G\backslash B)\leq P(G).\qed
\]
\end{lem} 

\begin{thm}\label{For block-cycle Z(G)=P(G)}
Let $G$ be a block-cycle graph. Then 
\[
Z(G)=P(G).
\]
Furthermore, the paths in any minimal path covering of $G$ are
precisely the forcing chains in a minimal zero forcing process
initiated by an appropriate selection of the end-points of the paths in this
collection.
\end{thm}
\begin{proof}
  We prove this theorem by induction on the number of blocks in
  $G$ along with Lemma~\ref{pendant block}. The only block-cycle
  graphs with one block are either an edge or a cycle; the
  theorem is clearly true for these two graphs. Assume that the theorem is
  true for all block-cycle graphs $G'$ with fewer than $t$
  blocks.

  Let $G$ be a block-cycle graph, then according to Lemma~\ref{pendant
    block}, there is a pendant block, $B$, in $G$ which is connected
  to the other blocks in $G$ only through the vertex $u$. Let $G'$ be the
  graph formed by removing all the vertices of the block $B$, except $u$, from
  $G$. 

  The graph $G'$ has $t-1$ blocks, so by the induction hypothesis
  $Z(G')=P(G')$ and appropriately chosen end-points of the paths in a
  minimal path covering of $G'$ forms a zero forcing set.  By
  Lemma~\ref{P(G-B), G is a block-cycle}, we have $P(G')\leq P(G)$ and
  there are two possible cases to consider.

\medskip
\underline{Case 1:} There is a minimal path-cover $\PP$ for the graph $G'$ in
  which $u$ is the end-point of a path $P$. 
\smallskip

  First assume that $B$ is the edge $uv$. Then $G'$ is the
  graph formed by removing the pendant vertex $v$ from
  $G$. Since $u$ is an end-point of $P$ and $v$ is only connected to
  $u$, returning $B$ to $G'$ does not change the path cover number of
  the graph. By the induction hypothesis, the paths in the path-cover
  $\PP$ are the forcing chains of the forcing process initiated by the
  end-points of the paths in $\PP$. Also since $u$ is an end-point of
  $P$, we can assume that it does not perform any forces. Therefore, the
  zero forcing process will be continued by using $u$ to force
  $v$. Thus,
\[
P(G')=P(G)\leq Z(G)\leq Z(G')=P(G'),
\]
which implies $Z(G)=P(G)$.

Next assume that $B$ is a pendant cycle. Then
\[
P(G') + 1  \leq P(G), 
\]
since at least two paths are needed to cover the vertices of a cycle.
Let $v$ and $w$ be the two neighbours of $u$ in $B$. Then, since $u$
is an end-point of $P$, the induced path $P' = P\cup \left(V(B) \backslash
  \{v\}\right)$ covers all vertices in the cycle $B$ except $v$.
Thus  $(\PP \setminus \{P\}) \cup \{P', \{v\}\}$ is a path cover for $G$ and
$P(G)=P(G')+1$. 

We also need that these paths are forcing chains. By assigning colour black to the
vertex $v$, all vertices in $B$ will be coloured by continuing the
forcing process in $P$ through $u$. Thus,
\[
P(G')+1=P(G)\leq Z(G)\leq Z(G')+1=P(G')+1,
\]
which implies $Z(G)=P(G)$.

\medskip
\underline{Case 2:} In every minimal path covering $\PP$ of $G'$, the
vertex $u$ is an inner vertex of a path $P$ in $\PP$.
\smallskip

Again, we first assume that $B$ is an edge $uv$.  Since $\PP\cup
\{v\}$ covers all the vertices of $G$, we have that $P(G) \leq
P(G')+1$.  But if $P(G) = P(G')$, then $v$ is covered in the same path
as $u$ in a path covering of $G'$; this contradicts the fact that $u$
is not an end-point of any path in any path covering of $G$. Thus
$P(G)=P(G')+1$.

If the vertex $v$ is assigned the colour black, then we are 
  able to colour the graph $G$ following the same forcing process
  which we followed to colour the graph $G'$. Thus,
\[
P(G')+1=P(G)\leq Z(G)\leq Z(G')+1=P(G')+1,
\]
which implies $Z(G)=P(G)$.

If $B$ is a cycle, then $P(G)= P(G')+1$, since $\PP$ along with a path
covering the vertices of $B\backslash\{u\}$, covers all vertices of
$G$. 

Let $w$ be a vertex in $B$ that is a neighbour of $u$.  The set of
initial set of black vertices in the zero forcing set of $G'$ along
with $w$ forms a zero forcing set for $G$. Thus,
\[
P(G')+1=P(G)\leq Z(G)\leq Z(G')+1=P(G')+1,
\]
which implies $Z(G)=P(G)$.
\end{proof}

The following corollary is obtained from the fact that any unicyclic
graph is a block-cycle graph. This result was also recently shown to be true in~\cite{MR2917419}.
\begin{cor}\label{unicyclics}
 If $G$ is a unicyclic graph, then $Z(G)=P(G)$.\qed
\end{cor}

It seems that for general graphs it is rare to have the equality
$Z(G)=P(G)$. To show this, along with the fact that the discrepancy
between $Z(G)$ and $P(G)$ can be arbitrarily large, we focus on the
family of graphs with $P(G)=2$.
\begin{prop}\label{prop:allkvalues}
  Let $G$ be a graph with $P(G)=2$ and two covering paths $P_1$ and
  $P_2$ with $|P_1|=m$ and $|P_2|=n$.  Then
\[
2 \leq Z(G)\leq \min \{n,m\}+1.
\]
Moreover, for any number $k$ in this interval, there is a graph $G$
satisfying $P(G)=2$ with $Z(G)=k$.
\end{prop}
\begin{proof}
  Assume that $m\leq n$ and note that the claim that $Z(G)\geq 2$ is trivial. 

  Let $B$ be the set consisting of $V(P_1)$ and an end-point of $P_2$.
  Obviously $B$ is a zero forcing set for $G$. Thus $Z(G)\leq |B|=m+1\leq \min
  \{n,m\}+1$.

  Let $k$ be any number in the interval $\{1,\dots, \min\{m,n\}\}$
  and $P_1$ and $P_2$ be two paths with $|P_1|=m$ and
  $|P_2|=n$. Starting with an end-point of $P_1$ make each of the
  first $k$ consecutive vertices of $P_1$ adjacent to all of the
  vertices of $P_2$.  Then, it is easy to observe that forcing number
  of this graph is $k+1$ (the forcing set consists on the consecutive
  $k$ vertices of $P_1$ along with an end-point of $P_2$).
\end{proof}

In the next section we discuss a family of graphs for which $P(G)=Z(G) =2$.

\section{Double paths}
\label{doublepaths}

A graph is \textsl{outerplanar} if there exists a planar embedding of
the graph in which all the vertices are contained in a single face.
If an outerplanar graph, that is not a path, has the property that its
vertices can be covered with two induced paths, then the graph is
called a \textsl{double path} or a \textsl{parallel path}. Such graphs
were also called a \textsl{graph of two parallel paths} in
\cite{johnson2009graphs}. We refer to these two induced paths,
naturally, as the \textsl{covering paths}.  Clearly if $G$ is a double
path, then $P(G)=2$. We will show that $Z(G)$ is also $2$, so double
paths are another family of graphs for which the path cover number
equals the zero forcing number.

\begin{thm}\label{ZdoublePath}
If $G$ is a double path, then $Z(G)=2$.
\end{thm}
\begin{proof}
  Assume that $P_1$ and $P_2$ are the two covering paths of $G$ in a
  given embedding of $G$.  Thus the paths are fixed and we can talk
  about the left end points of the paths (or, equivalently, the right
  end points).  Let $u$ be the left end point of $P_1$ and $v$ the
  left end point of $P_2$. It is not difficult to deduce that since no
  edges are cross in this graph, the set $\{u,v\}$ forms a minimal
  zero forcing set for $G$.
\end{proof}

Note, if we choose the left endpoints $\{u,v\}$ of a double path, then
it follows that $P_1$ and $P_2$ (as in the proof above) will also form
the corresponding zero forcing chains for a double path.

The concept of a double path can be generalized. If the vertices of a graph $G$ can be
partitioned into paths $P_1,P_2,\dots ,P_k$ so that:
\begin{enumerate}
\item the only vertices not in the path $P_i$ that are adjacent to a vertex in $P_i$ are in either $P_{i-1}$ or $P_{i+1}$ 
(assume $P_0$ and $P_{k+1}$ are the empty set), and
\item  the graph induced by $P_{i}$ and $P_{i+1}$ is a double
path for $i =1,\dots, k-1$, 
\end{enumerate}
then $G$ is called a \textsl{series of parallel paths}.

\begin{thm}\label{pathseries}
  If $G$ is a series of parallel paths, then $Z(G) =P(G)$. Moreover,
  the left (or right) endpoints of the covering paths form a zero
  forcing set.
\end{thm}
\begin{proof}
  We use induction on the number of paths. If there are two paths,
  then the result follows from Theorem~\ref{ZdoublePath} and assume
  the result holds for any series of parallel paths with $k$ paths.

  Let $G$ be a graph that is the series of $k+1$ parallel paths and
  assume that $P_{k+1}$ is the final path. Set $G' = G \backslash
  (P_{k+1})$. The set of left end points of the path covering of $G'$ forms a zero
  forcing set of $G'$. Then these end points, together with
  the left end point of $P_{k+1}$ forms a zero forcing set for $G$ and
  the forcing chains are the paths,  $P_1,P_2,\dots ,P_{k+1}$.
\end{proof}

We note here that Theorem \ref{pathseries} may be used to yield the zero 
forcing number of the grid, namely the cartesian product of two paths (see also
\cite{aim}). That is, the zero forcing number of the $m$-by-$n$ grid is given by
$\min \{ m,n \}$.

Theorem~\ref{ZdoublePath} can also be obtained from \cite[Theorem
5.1]{johnson2009graphs}, although zero forcing is not considered in \cite{johnson2009graphs}.  
In fact, they show that among all the graphs
with $P(G)=2$, only those that are also outerplanar satisfy M$(G)=2$.

There are outerplanar graphs for which the path cover number and the
zero forcing number are arbitrarily far apart (see, for example,
\cite[Ex. 2.11]{MR2645093}).  Motivated by this, in the next section
we consider positive zero forcing sets, and the positive zero forcing
number.

\section{Positive zero forcing sets}
\label{forcing_trees}

In 2010, a variant of the zero forcing number, called 
positive semi-definite zero forcing or the positive zero forcing
number, was introduced in \cite{MR2645093}, and a collection of its properties
were discussed in \cite{ekstrand2011positive} and
\cite{ekstrand2011note}.  The positive zero forcing number is also
based on a colour change rule that is very similar to the zero forcing
colour change rule.  Let $G$ be a graph and $B$ a set of vertices; we
will initially colour of the vertices of $B$ black and all other
vertices white. Let $W_1,\dots,W_k$ be the set of vertices of the
connected components of $G\backslash B$.  If $u$ is a vertex in $B$
and $w$ is the only white neighbour of $u$ in the graph induced by
$V(W_i\cup B)$, then $u$ can force the colour of $w$ to black. This is
the \textsl{positive colour change rule}.  The definitions and terminology for
the positive zero forcing process, such as, colouring, derived set,
positive zero forcing number etc., are identical to those for the zero
forcing number, except we use the positive semi-definite colour change
rule.

The size of the smallest positive zero forcing set of a graph $G$ is
denoted by $Z_+(G)$. Note that for any non-empty graph
\[
1\leq Z_+(G) \leq |V(G)|-1;
\]
the lower bound holds with equality if and only if $G$ is a tree and
the upper bound only for complete graphs. Also for all graphs $G$,
since a zero forcing set is also a positive zero forcing set we have
that $Z_+(G) \leq Z(G)$. Moreover, in \cite{MR2645093} it was observed that
$\Mp(G) \leq Z_{+}(G)$, for any graph $G$.

We have seen that applying the zero forcing colour change rule to the
vertices of a graph produces a path covering for the
graph. Analogously, applying the positive colour change rule produces
a set of induced trees in the graph, and we refer to these trees as
\textsl{forcing trees}.  To define these trees, let $G$ be a graph and
$Z_p$ be a positive zero forcing set of $G$. Construct the derived
set, recording the forces in the order in which they are performed;
this is the chronological list of forces.  Note that in applying the
colour change rule once, two or more vertices can perform forces at
the same time and a vertex can force multiple vertices from different
components at the same time.  For any chronological list of forces, a
forcing tree is an induced rooted tree, $T_r$, formed by a sequence of
sets of vertices $(r,X_1,\dots, X_k)$, where $r\in Z_p$ is the root
and the vertices in $X_i$ are at distance $i$ from $r$ in the
tree. The vertices of $X_i$ for $i=1,\dots, k$, are forced by applying
the positive semi-definite colour change rule with the vertices in
$X_{i-1}$; so for any $v \in X_i$ there is a $u \in X_{i-1}$, such
that $u$ forces $v$ if and only if $v$ is a neighbour of $u$.  

In a forcing tree, the vertices in $X_i$ are said to be the vertices
of the $i$-th level in the tree. Note that the vertices in a specific
level may have been forced in different steps of the positive
semi-definite colour change procedure and they may also perform forces
in different steps.

\begin{example} 
  The positive zero forcing number of the graph $G$ in
  Figure~\ref{ex.forcing trees} is three and $\{1,3,10\}$ is a positive zero forcing set for
  the graph. The forcing trees in this colouring procedure (as
  depicted in Figure~\ref{ex.forcing trees}) are as follows: \newline
  $T_1=\{1,X_1,X_2,X_3\}$, where $X_1=\{2,6\}$, $X_2=\{5,7\}$, $X_3=\{4,8\}$,\\
  $T_{3}=\{3\}$.
  $T_{10}=\{10 ,X_1\}$, where  $X_1=\{9\}$.\\
\begin{figure}
\centering
\vspace{.5cm}
\includegraphics[width=.6\textwidth]{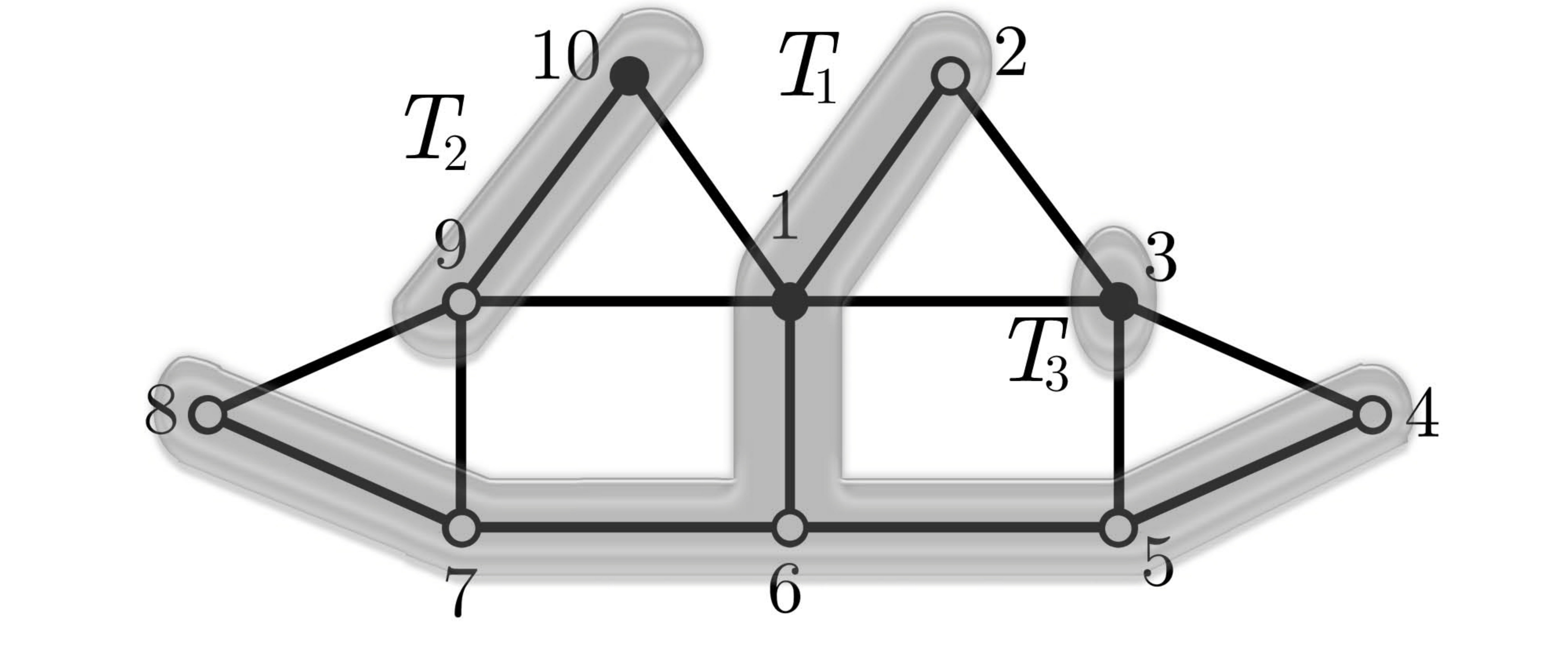} \\
\caption{The forcing trees in a graph}
\label{ex.forcing trees}
\end{figure}
\end{example}

Our next concept is analogous to the path cover number of a graph.  A
\textsl{tree covering of a graph }is a family of induced vertex
disjoint trees in the graph that cover all vertices of the graph. The
minimum number of such trees that cover the vertices of a graph $G$ is
the {\em tree cover number} of $G$ and is denoted by $\T(G)$. Any set of
forcing trees corresponding to a minimal positive zero forcing set is of size $Z_+(G)$ and
covers all vertices of the graph. This implies the following.
\begin{prop}\label{T(G)&Z_+(G)}
For any graph $G$, we have $\T(G)\leq Z_+(G)$.\qed
\end{prop}

This bound is clearly tight for trees and cycles. But there are
graphs, such as complete bipartite graphs, for which the discrepancy
between these parameters can be arbitrarily large. It is, hence, an
interesting question to ask for which families of graphs does equality
hold between these two parameters. One way to approach this problem is
to find a graph operation which preserves the equality in graphs for
which these parameters agree. Using this approach, in the next section
we will prove that equality between these two parameters holds for
block-cyclic graphs.  In Section~\ref{doubletrees} we will define a
family of graphs called \textsl{double trees}, these are analogous
to the double paths of Section~\ref{doublepaths} and we will show that
the positive zero forcing number and the tree cover number of these
graphs coincide. We will also show these parameters are equal for a
much larger family of graphs in Section~\ref{outerplanar}.

\section{Vertex-sum of graphs}
\label{vertex_sum_of_two_graphs}

Let $G$ and $H$ be two graphs and assume that $v$ is a vertex of both
$G$ and $H$, then the \textsl{vertex sum of $G$ and $H$ over $v$} is
the graph formed by identifying $v$ in the two graphs.  The vertex-sum
is denoted by $G \stackplus{v} H$. 

A block-cycle graph can be
recursively defined as the vertex sum of a block-cycle graph with
a cycle or a path. The next result shows how to calculate the tree cover
number of the vertex sum of two graphs.

\begin{lem}\label{Tofvertexsum}
For any graphs $G$ and $H$, both with an identified vertex $v$, we have 
\[
T(G\,\stackplus{v}\,H)=T(G)+T(H)-1.
\]
\end{lem}
\begin{proof}
  Let $\TT_G$ and $\TT_H$ be minimal tree coverings of $G$ and
  $H$, respectively, and suppose $T_1\in \TT_G$ and $T_2\in \TT_H$ are the
  trees covering $v$. Let $T_v=T_1 \,\stackplus{v}\, T_2$. Observe
  that $T_v$ is an induced tree in $G\,\stackplus{v}\,H$ that also covers
  $v$. Therefore,
\[
\left(\left(\TT_G  \cup \TT_H\right) \backslash \{T_1,T_2\}\right)\cup T_v
\]
is a tree covering of $G\,\stackplus{v}\,H$.

Next we show, in fact, the vertices of $G\,\stackplus{v}\,H$ can not
be covered with fewer trees.  Suppose that the vertices of
$G\,\stackplus{v}\,H$ can be covered with $T(G)+T(H)-2$ induced
disjoint trees.  Let $T$ be the tree that covers $v$ in such a tree
covering of $G\,\stackplus{v}\,H$. At least $T(G)-1$ trees are needed
to cover the vertices of $G\backslash T$. Since the number of trees in
the covering of $G\,\stackplus{v}\,H$ is $T(G)+T(H)-2$, the vertices
of $H\backslash T$ are covered by at most ($T(H)-2$) trees, but this
contradicts the fact that $T(H)$ is the least number of
trees that cover the vertices of $H$. 
\end{proof}

\begin{lem}\label{Z+ofvertexsum}
For any graphs $G$ and $H$, both with an identified vertex $v$, we have 
\[
Z_+(G\,\stackplus{v}\,H)=Z_+(G)+Z_+(H)-1.
\]
\end{lem}
\begin{proof}
To prove this equality, let $\TT_G$ and $\TT_H$ be the sets of
forcing trees for a minimal positive zero forcing set in $G$ and $H$ respectively. Let
$T_1\in \TT_G$ and $T_2\in \TT_H$ be the trees that contain $v$. Then
$T_1\,\stackplus{v}\,T_2$ is a forcing tree in $G\,\stackplus{v}\,H$
covering $v$. Then similar reasoning as in the previous lemma applies.
\end{proof}

\begin{cor}\label{Z+vertexsum}
If $G$ and $H$ are two graphs that satisfy $Z_+(G)=T(G)$ and $Z_+(H)=T(H)$, then 
\[
Z_+(G\,\stackplus{v}\,H)=T(G\,\stackplus{v}\,H),
\] where $v$ is an identified vertex in both $G$ and in $H$. \qed
\end{cor}

Since a block-cycle graph is the vertex sum of a block-cycle graph and
either a cycle or a path we have the following result.

\begin{cor}
If $G$ is a block-cycle graph, then $Z_{+}(G) = T(G)$. \qed
\end{cor}

\section{Double trees}
\label{doubletrees}
In the next section, we will show that $Z_+(G) = T(G)$ for every
outerplanar graph. The first step towards verifying this claim is to show that it
holds for a subset of these graphs called \textsl{double trees}.
Recall that in Section~\ref{doublepaths} we defined a double path, whereas a
double tree can be viewed as an extension of the concept of a double
path. If the vertices of a connected outerplanar graph, which is not a
tree, can be covered with two induced trees, then the graph is called
a \textsl{double tree}.

Our first step will be to show that if $G$ is a double path, then
$Z_+(G) = T(G) = 2$. Then any double tree can be constructed by applying
an appropriate series of vertex sums of trees with an appropriate
double path.  Thus, from Corollary~\ref{Z+vertexsum}, we will be able
to conclude that $Z_+(G) = T(G) = 2$ holds for any double tree.

In the following three lemmas we find different positive zero forcing
sets for double paths.  In all of these lemmas we will assume that $G$
is a double path with a specific planar embedding of $G$ with covering
paths $P_1$ and $P_2$. Since this planar embedding of $G$ is fixed, we
can refer to the end points of a covering path as the right end point
and the left end point.

\begin{lem}\label{lemma1} 
  If $u$ and $v$ are both right (or both left) end points of $P_1$ and
  $P_2$, respectively, then $\{u,v\}$ is a positive zero forcing set
  of $G$. Moreover, if $\{u,v\}$ are initially coloured black, then
  there is a positive zero-forcing process in which the forcing trees are $P_1$
  and $P_2$.
\end{lem}
\begin{proof} 
  We prove this lemma by induction on the number of vertices. The result is
  true for $C_3$, which is a connected double path on the fewest number of
  vertices.  Assume that it is true for all graphs $H$ with
  $|V(H)|<n$. 

  Let $G$ be a graph on $n$ vertices. Assume that $u$ and $v$ are left
  end points of $P_1$ and $P_2$, respectively. By assigning the colour
  black to each of these vertices we claim that $\{u,v\}$ is a positive zero forcing set of
  $G$. If $u$ is a pendant vertex, then it forces its only neighbour,
  say $w$ (which must be in $P_1)$, which is a left end point of a
  covering path in $G\backslash u$. Thus by the induction hypothesis
  $\{w,v\}$ is a positive zero forcing set of $G$ and there is a positive zero forcing process in
  which the forcing trees are $P_1$ and $P_2$. Similarly if $v$ is a
  pendant vertex, using a similar reasoning, the lemma follows.  

  If neither $u$ nor $v$ are pendant, then $u$ and $v$ are adjacent
  and since both are on the same side, at least one of them, say $u$,
  is of degree two. Let $w$ be the only neighbour of $u$ in
  $P_1$. Thus $u$ can force $w$ (its only white neighbour) and again
  by the induction hypothesis $\{w,v\}$ is a positive zero forcing set of $G$ and there is
  a positive zero forcing process in which the forcing trees are $P_1$ and
  $P_2$. The same reasoning applies when $u$ and $v$ are the right end
  points of $P_1$ and $P_2$, respectively.
\end{proof}

\begin{lem}\label{lemma2}
  If $u$ and $v$ are two vertices of $P_1$ and $P_2$, respectively,
  which form a cut set for $G$, then $\{u,v\}$ is a positive zero
  forcing set for $G$. Moreover, there is a positive  zero forcing forcing
  process in which $P_1$ and $P_2$ are the forcing trees (with $u$ and
  $v$ the roots of the tree).
\end{lem}
\begin{proof}
  Let $W_1\subseteq V(G)$ and $W_2\subseteq V(G)$ be the vertices of
  the left hand side and the right hand side components of
  $G\backslash \{u,v\}$, respectively, and let $G_1$ and $G_2$ be the
  subgraphs induced by $\{u,v\}\cup W_1$ and $\{u,v\}\cup W_2$,
  respectively. Then according to Lemma~\ref{lemma1}, $\{u,v\}$ is a
  positive zero forcing set for both $G_1$ and $G_2$ and thus a positive 
zero forcing set for $G$ and there is a positive zero forcing process in which $P_1$ and $P_2$ are the forcing trees.
\end{proof}

\begin{lem}\label{lemma3}
  If $u$ is a vertex in $P_1$ which is not an end point, then there
  always is a vertex $v$ in $P_2$ such that $\{u,v\}$ is a cut set of
  $G$.
\end{lem}
\begin{proof} If $P_2$ contains at most two vertices this result is clear since $u$ is 
not an end point.
  Suppose there is a vertex $u$ in $P_1$ for which there is no vertex
  $v$ in $P_2$ such that $\{u,v\}$ is a cut set for $G$. This implies
  that $|P_2| \geq 3$ since if $P_2=\{v_1,v_2\}$, then at least one of
  $\{u,v_1\}$ and $\{u,v_2\}$ is a cut set.

  Let $v$ be any non-pendant vertex of $P_2$. Obviously $u$ is a cut
  vertex of $P_1$ and $v$ is a cut vertex of $P_2$. Since $\{u,v\}$ is
  not a cut set of $G$, there is a vertex in the left hand side (or
  right hand side) of $u$ that is adjacent to a vertex in right hand
  side (or left hand side) of $v$. Assume that $w$ is the farthest
  vertex from $v$ in $P_2$ having this property. Since $w$ is the
  farthest vertex from $v$ with the described property and $G$ is an
  outerplanar graph, $\{u,w\}$ is a cut set of $G$ which contradicts
  with the fact that, there is no vertex in $P_2$ that forms a cut set
  along with $u$ for $G$.
\end{proof}

Combining Lemmas \ref{lemma1}, \ref{lemma2} and \ref{lemma3} along
with the fact that in the proof of all these three lemmas forces are
performed along the covering paths, we have the following.

\begin{cor}\label{PZFS of double path}
  Let $G$ be a double path with covering paths $P_1$ and $P_2$. Then
  for any vertex $v$ in $P_1$, there is always another vertex $u$ in
  $P_2$ such that $\{u,v\}$ is a positive zero forcing set for
  $G$. Moreover there is a positive zero forcing process in which the two paths
  $P_1$ and $P_2$ are a minimal set of forcing trees in $G$.\qed
\end{cor}

The following result is a consequence of Corollary~\ref{PZFS of double path}.
\begin{cor}\label{PZFS of double tree}
  Let $G$ be a double tree with covering trees $T_1$ and $T_2$.  Then
  for any vertex $v$ in $T_1$, there is always another vertex $u$ in
  $T_2$ such that $\{u,v\}$ is a positive zero forcing set for
  $G$. Moreover $\{T_1,T_2\}$ coincides with a minimal collection of
  forcing trees in $G$.\qed
\end{cor}

\section{Outerplanar graphs}
\label{outerplanar}

In \cite{barioli2011minimum} it is shown that the maximum positive
semi-definite nullity is equal to the tree cover number for any
outerplanar graph. Since the positive zero forcing number is an upper
bound on the maximum positive semi-definite nullity, this implies that $Z_+(G)
= T(G)$ holds for any outerplanar graph.  This result was shown to be
true in \cite{barioli2011minimum} and \cite{ekstrand2011note} (where
the proof is generalized to 2-trees), in this section we give a
different proof of this fact that does not rely on Schur-complements
or orthogonal removal.  Moreover, we show that any minimum tree
covering of an outerplanar graph coincides with a minimum collection
of zero forcing trees.

In a fixed tree covering of a graph, two trees, $T_1$ and $T_2$, are
said to be \textsl{adjacent} if there is at least one edge $uv\in
E(G)$ such that $v \in V(T_1)$ and $u \in V(T_2)$.  A tree is called
\textsl{pendant} if it is adjacent to only one other tree from this
given tree covering.

Throughout this section, $G$ will be an outerplanar graph with a
planar embedding in which all the vertices are on the same face.  An
edge of $G$ is called \textsl{outer} if it lies on the face containing
all of the vertices; it an edge is not outer, then it is called
\textsl{inner}.  Further, let $\TT(G)$ be a minimal tree covering for
$G$.  Define $H_{\TT}$ to be the graph whose vertices correspond to
the trees in $\TT(G)$ and two vertices in $H_{\TT}$ are adjacent if
there is an outer edge between the corresponding trees in the graph
$G$.  Two trees of $\TT(G)$ are called \textsl{consecutive}, if their
corresponding vertices in $H_{\TT}$ are adjacent vertices each of
degree two.

\begin{thm}\label{consecutive trees}
  Let $G$ be an outerplanar graph and $\TT(G)$ a minimum tree covering
  for $G$. If there is no pendant tree in $\TT(G)$, then there are at
  least one pair of consecutive trees in $\TT(G)$.
\end{thm}
\begin{proof}
  Assume that $\TT(G)$ is a minimum tree covering for $G$ in which
  there is no pendant tree. Then, two cases are possible:

  {\bf Case 1.} There is no tree in $\TT(G)$ with at least one of the
  inner edges of $G$ in its edge set (see Figure~\ref{fig_for_H_T} for
  an example of such a graph). Therefore $H_{\TT}$ is a
  cycle. Accordingly, any adjacent pair of trees in $\TT(G)$ are
  consecutive.
\begin{figure}[ht!]
\vspace{.5cm}
\centering
\includegraphics[width=.6\textwidth]{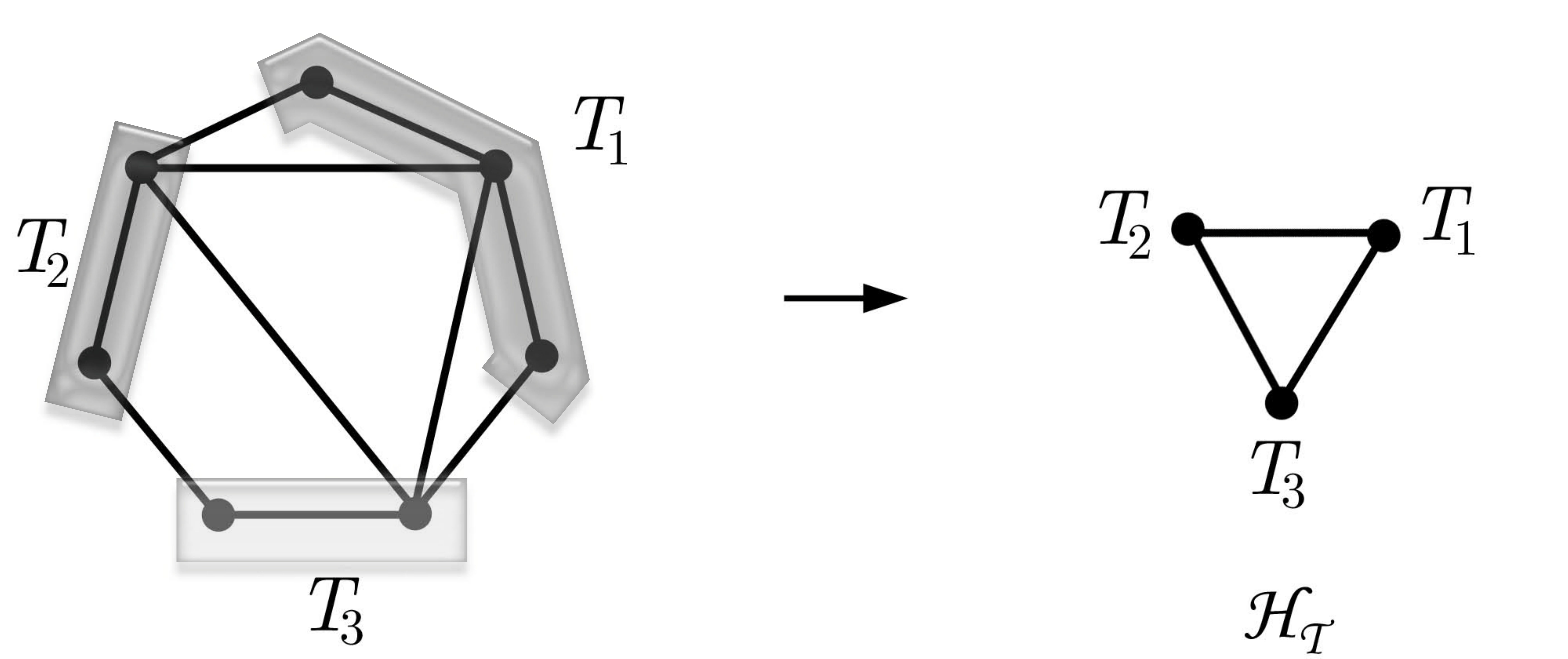}
\caption{Forcing trees with no inner edge in their edge set}
\label{fig_for_H_T}
\end{figure}

{\bf Case 2.}  There is at least one tree in $\TT(G)$ that has an
inner edge of $G$ in its edge set (Figure~\ref{fig_for_H_T_2} gives an
example of such a graph). 

\begin{figure}[ht!]
\vspace{.5cm}
\centering
\includegraphics[width=.8\textwidth]{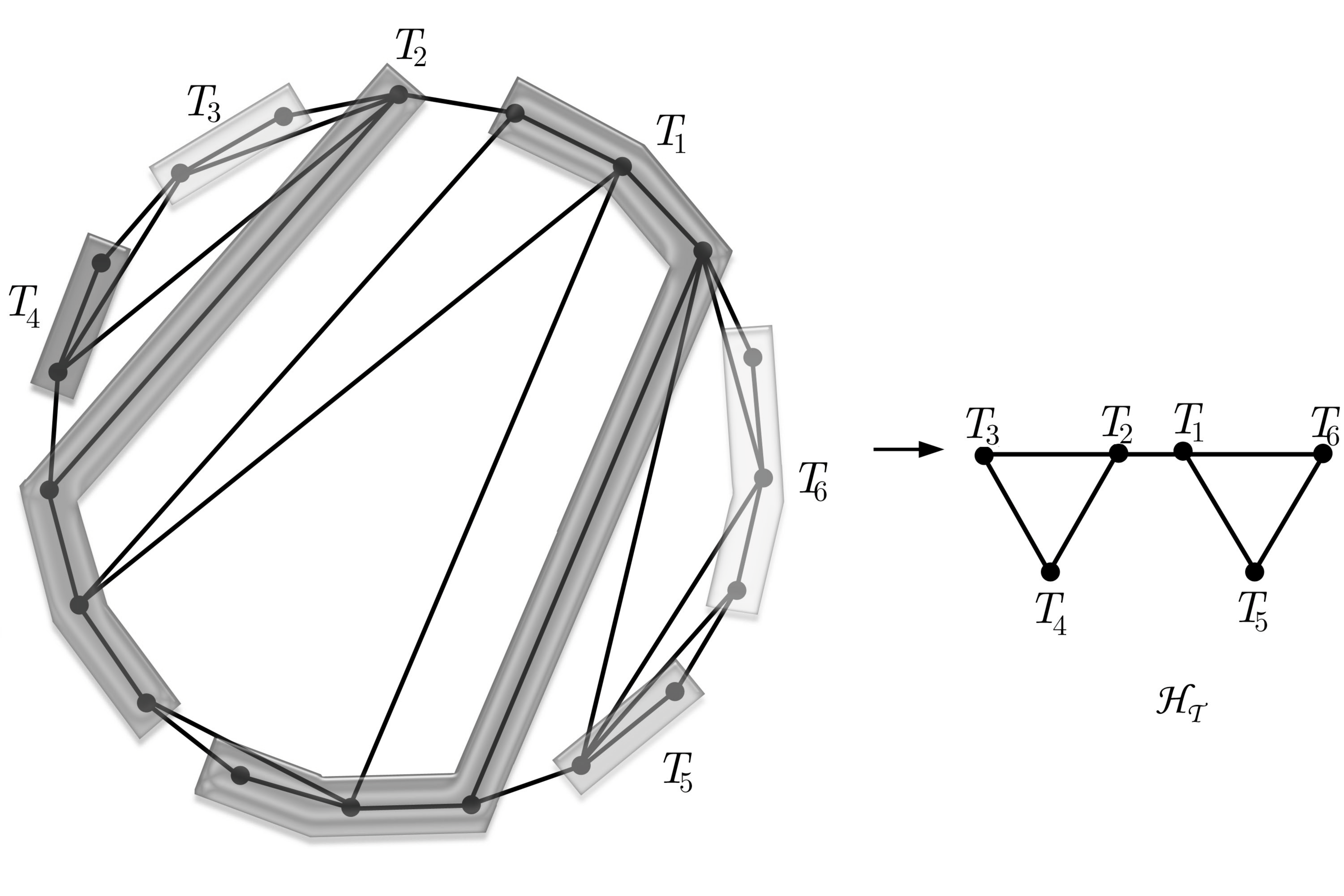}
\caption{$T_1$ is a forcing tree with an inner edge in its edge set}
\label{fig_for_H_T_2}
\end{figure}

The idea in this case is that we will select ``left-most'' such inner
edge of $G$. Then the subgraph induced by this edge, and all the
vertices to the left of the edge, form an outerplanar graph in which no
tree includes an inner edge of $G$. In case 1, we showed that such a
subgraph will have a consecutive pair of trees, and thus so will $G$.

First, let $W$ be the set of all vertices which are the end-points of
an inner edge of $G$ that is also included in the edge set of a tree
in $\TT(G)$. Second, note that any inner edge $e = \{u,v\}$ of $G$
partitions the plane into two parts and, consequently, partitions the
set of vertices of $G \backslash \{u,v\}$ into two subsets $V_e'$ and
$V_e''$. Since $G$ is outerplanar and finite, there exists an inner
edge, $e$ in $W$ such that at
least one of $V_e'$ or $V_e''$ does not contain any of the vertices in
$W$. We will assume that $V_e'$ is the vertex set that is disjoint
from $W$.

The subgraph of $G$ induced by the vertices $V_e' \cup \{u,v\}$ is an
outerplanar graph; call this $H$.  The trees in $\TT(G)$ that
intersect with $H$ form a minimal tree covering of $H$, which we will
call $\TT(H)$. None of the trees in $\TT(H)$ can include an inner edge
of $H$ (this follows as there are no vertices that are both in $H$ and
in $W$). By Case 1, there is a pair of adjacent trees in $\TT(H)$ and
therefore there is at least one pair of consecutive trees in $\TT(G)$.
\end{proof}

The following theorem shows that for any outerplanar graph, it is
possible to find a minimal tree covering that has a pendant tree; this
plays a key role in the proof of the fact that outerplanar graphs
satisfy $Z_+(G)=T(G)$.

\begin{thm}\label{pendant tree}
  Let $G$ be an outerplanar graph. Then there is a minimum tree
  covering for $G$ in which there is a pendant tree.
\end{thm}
\begin{proof}
  Assume that $\TT(G)$ is a minimum tree covering for $G$ in which
  there is no pendant tree. We use Theorem~\ref{consecutive trees} to
  construct a new tree covering $\TT'(G)$ of $\TT(G)$ with
  $|\TT'(G)|=|\TT(G)|$ in which there is a pendant tree.

  By Theorem~\ref{consecutive trees} there are two trees $T_1$ and
  $T_2$ in $\TT(G)$, which are consecutive. Let $H$ be the outerplanar
  graph induced by $V(T_1)\cup V(T_2)$. There are two outer edges in
  $H$ that have an end-point from each of trees $T_1$ and $T_2$
  (otherwise $T_1 \cup T_2$ would be a tree and $\TT(G)$ would not be
  a minimum tree covering). One of these outer edges of $H$, call it
  $e=\{u,v\}$, is an inner edge in $G$; we will assume that $u \in
  T_1$ and $v \in T_2$.

Similarly, if $u$ has any neighbour in $T_2$, then $v$ has no other
neighbours in $T_1$. Thus we will assume that $v$ is the only neighbour of $u$ in $T_2$.
In fact, the subgraph $T_1 \backslash \{u\}$
  is a forest and exactly one of the trees in the forest has vertices
  which are adjacent to a vertex in $T_2$. Call this tree
  $S_1$. Define a second new tree by
\[
S_2 = \left( T_2 \stackplus{v} \{u,v\} \right) \stackplus{u} (T_1 \backslash S_1).
\]
By replacing $T_1$ and $T_2$ in $\TT(G)$ with $S_1$ and $S_2$ we can
construct a new minimum tree covering for $G$ in which $S_1$ is a
pendant tree.
 
  A similar argument applies when $u$ has another neighbour in $T_2$.
  If neither $u$ nor $v$ has any other neighbour in $T_1$ and $T_2$,
  respectively, then either case mentioned above are applicable.
\end{proof}

We now have all the necessary tools to prove the main result of this section.

\begin{thm}\label{outerplanars satisfy Z_+=T}
Let $G$ be an outerplanar graph. Then
\[
Z_+(G)=\T(G).
\]
Moreover, any minimal tree covering of the graph $\TT(G)$ coincides with a
collection of forcing trees with $|\TT(G)|=Z_+(G)$.
\end{thm}
\begin{proof}
  We prove the claim by induction on the tree cover number and using
  Theorem~\ref{pendant tree}. It is obviously true for $\T(G)=1$. Assume
  that it is true for any outerplanar graph $G'$ with $\T(G') <
  k$. Now let $G$ be an outerplanar graph with $\T(G)=k$. By
  Proposition~\ref{T(G)&Z_+(G)}, we have $ Z_+(G) \geq
  \T(G)$. 

  Let $\TT(G) = \{T_1,T_2, \dots, T_k\}$ be a minimum tree covering of
  $G$.  We first consider the case when $\TT(G)$ contains a pendant tree.

 {\bf Case 1.} Assume that $T_1$ is a pendant
  tree and that $T_2$ is the only tree adjacent to $T_1$. Let $G'$ be
  the graph induced by the vertex set $V(G) \backslash V(T_1)$, then
  the induction hypothesis holds, so $T(G')=Z_+(G')=k-1$. Further,
  $T_2, T_3, \dots, T_{k}$ are forcing trees in a positive zero
  forcing process whose initial set of black vertices, $Z'_p$, has a
  vertex from each tree in $\TT(G)\backslash T_1$. 

  A positive zero forcing process in $G$ starting with the black vertices in
  $Z'_p$ can proceed as it does in $G'$ until the first vertex of
  $T_{2}$, say $x$, that is adjacent to some vertex in $T_1$ gets
  forced. Since the graph induced by $V(T_1)\cup V(T_2)$ is a double
  tree, according to Corollary \ref{PZFS of double tree}, the vertex
  $x$ determines a vertex $y$ in $T_1$ such that $\{x,y\}$ is a
  positive zero forcing set for the subgraph induced by $V(T_1)\cup
  V(T_2)$. Since the induction hypothesis holds for this subgraph, the
  tree $T_2$ is a forcing tree in this subgraph as well. Thus the
  vertices of $T_2$ get forced in the same order as they were forced in
  $G'$.

  Therefore we can complete the colouring of $G$ by adding the vertex
  $y$ to the initial set of black vertices.  Thus, $Z_p=Z'_p\cup
  \{y\}$ is a positive zero forcing set of $G$ with $T_1,T_2, \dots,
  T_k$ as the forcing trees in this positive zero forcing
  process. Thus, $Z_+(G)= \T(G)$.

  {\bf Case 2.} If $\TT(G)$ does not contain a pendant tree, then by
  Theorem~\ref{pendant tree}, it is possible to build a new tree
  covering that does have a pendant tree.  By Case 1, this new tree
  covering has exactly $Z_+(G)$ trees and the trees of this new tree
  covering are forcing trees. Now we need to show that the original
  minimal tree covering of the graph $G$ also coincides with a
  collection of forcing trees associated with a positive zero forcing
  set of $G$.

  We will assume that $T_1$ and $T_2$ are a pair of consecutive trees
  in $\TT(G)$ (from Theorem~\ref{consecutive trees} we know that such
  a pair exists). Using the same notation as in Theorem~\ref{pendant tree} we
  assume that $v\in T_2$ has a neighbour, other than $u$, in $T_1$
  (the other cases are similar).  In the procedure of constructing a
  pendant tree in the minimal tree covering of $G$, the pair of
  consecutive trees $T_1$ and $T_2$ were modified to obtain two new
  trees called $S_1$ and $S_2$.  The tree $S_1$ is pendant and only
  adjacent to the tree $S_2$ in the new minimal tree covering. Let
\[
\TT(G)' = (\TT(G) \backslash \{T_1, T_2\}) \cup \{S_1 , S_2\},
\]
and define 
\[
 T_1' = T_1 \backslash S_1, \qquad T_1'' = S_1.
\]
We will use a similar decomposition of $T_2$.  Let $T_2''$ be the set
of all the trees in the forest $T_2 \backslash \{v\}$ that have a
vertex which is adjacent to some vertex in $T_1''$. Define $T_2' = T_2
\backslash T_2''$.

%

Let $Z'_p$ be a positive zero forcing set for $G$ for which $\TT(G)'$ is a set
of zero forcing trees. Assume that $x \in V(S_2)$ and $y \in V(S_1)$
are the two vertices in $Z_p'$ from the trees $S_1$ and $S_2$ (from
Case 1 such the zero forcing set must have two such vertices).  We
will consider two cases, the first is when $x$ is a vertex in $T_2$
and second is when $x \in T_1$.

In the first case $x \in T_2$ and $y \in T_1$. We claim that $Z_p'$ is
a positive zero forcing set and there is positive zero forcing process in which the
trees of $\TT(G)$ are the zero forcing trees.  To see this we will
describe the positive zero forcing process. 

The positive zero forcing process proceeds along the forcing trees $S_1$ and
$S_2$ until the vertex $v$ is forced. Note that in the original
process, $v$ must force $u$. If the vertices $v$ and $y$ are removed
then one of the connected components will include all of $T_2''$ and
some of the vertices from $T_1''$. Starting with $v$ and $y$ it is possible to force all
the vertices in this component following the trees $T_2''$ and the
portion of $T_1''$ in the component. Once all the vertices of $T_2''$
are black, $y$ can force the remaining vertices along the tree
$T_1''$. Then the vertex in $T_1''$ that is adjacent to $u$ will force
$u$. Then $u$ can force the remaining vertices of $T_1'$.

In the second case both $x$ and $y$ are vertices in $T_1$. We claim that $(Z_p'
\backslash \{x\}) \cup \{v\}$ is a positive zero forcing set and there is positive zero
forcing process in which the trees of $\TT(G)$ are the zero forcing
trees.  Just as in the previous case, using $v$ and $y$ all the
vertices of $T_2''$ will be forced. Then, starting with $y$, all the
vertices for $T_1''$ will be forced with $T_1''$ the forcing tree.
Then the unique vertex in $T_1''$ adjacent to $u$ will force $u$ and
the positive zero forcing process will continue along $T_1'$. Finally, starting
with $v$, the vertices along $T_2'$ will be forced.
\end{proof}

\section{$k$-Trees}\label{k_trees}

A \textsl{$k$-tree} is constructed inductively by starting with a
$K_{k+1}$ and at each step a new vertex is added and this vertex is
adjacent to exactly $k$ vertices in an existing $K_{k}$. A
\textsl{partial $k$-tree} is any graph that is the subgraph of a
$k$-tree.  In particular, a graph is a partial $2$-tree if and only if
it does not have a $K_4$ minor (see \cite[p. 327]{Die}). Since outerplanar graphs are
exactly the graphs with no $K_4$ and $K_{2,3}$ minors (see \cite[p. 107]{Die}, it is easy to
see that every outerplanar graph is a partial $2$-tree.
In~\cite{ekstrand2011note} it is shown that the proof that the
maximum positive semi-definite nullity for outerplanar graphs is equal to the
tree cover number from~\cite{barioli2011minimum}, can be extended to
include any partial $2$-tree. From this it follows that if $G$ is a
partial $2$-tree, then $Z_+(G) = T(G)$.

In this section, we will give a purely graph theoretical version of
this result for a subset of $2$-trees. We also try to track the
variations between the positive zero forcing number and the tree cover
number in this subset of $k$-trees with $k>2$. This demonstrates that
$2$-trees are rather special when it comes to comparing $Z_+(G)$ and
$T(G)$. 

We will define a type $k$-tree that we call a \textsl{$k$-cluster}.  These
$k$-trees are constructed recursively starting with a $H=K_{k+1}$.  At
each step a new vertex is added to the graph and this new vertex is
adjacent to exactly $k$ of the vertices in $H$. In a general $k$-tree
the new vertices are adjacent to any $k$-clique in the graph, but in a
$k$-cluster the new vertices must be adjacent to a $k$-clique in $H$.
Observe that for each vertex $v$ not in $H$, there is exactly one
vertex in $H$ that is not adjacent to it. 

If $G$ is a $k$-cluster, then define $S(G)$ to be the set of all distinct
$k$-cliques $H' \subset H$ with the property that $H' \cup \{v\}$
forms a clique of size $k+1$ in $G$, for some $v\in V(G)\backslash
V(H)$.  The size of $S(G)$ can be no more than $k+1$.

\begin{thm}\label{Z_+&T of k-trees}
Suppose $G$ is a $k$-cluster and let $S(G)$ be as defined above. 
\begin{enumerate}
\item If $|S(G)|\geq 3$, then $Z_+(G)=k+1$.
\item If $|S(G)|<3$, then $Z_+(G)=k$.
\item If $|S(G)|=k+1$ and $k$ is  even, then $T(G)=\lceil \frac{k+1}{2} \rceil+1$.
\item If $|S(G)|<k+1$ and $k$ is even, then $T(G)=\lceil \frac{k+1}{2} \rceil$.
\end{enumerate}
\end{thm}   
\begin{proof}
  Since the minimum degree of $G$ is $k$, it is clear that $Z_+(G)\geq
  k$. Suppose $H$ is the initial $K_{k+1}$ in the
  $k$-cluster. Obviously, the set $V(H)$ is a positive zero forcing
  set for $G$ and $Z_+(G) \leq k+1$.
 
  To prove the first statement suppose $B$ is a positive zero forcing
  set of $G$ with $|B|=k$. Since $|S(G)| \geq 3$, there are three
  vertices $u,v,w \in V(G)\backslash V(H)$ that are adjacent to
  distinct $k$-sets in $H$. Thus any vertex $x \in V(H)$ is adjacent
  to at least two of these vertices. So no positive zero forcing set can be
  contained in $V(H)$.  Thus there must be a vertex $z\in B$ such that
  $z\not\in H$. If $z$ has only one white neighbour, then this
  neighbour must be in $H$ and $z$ can force it and we return to the
  first case. If $z$ has two or more white neighbours, then no vertex
  of $B$ can perform a force and again we reach a contradiction.

  To verify the second statement observe that if $|S(G)|<3$, then there
  is at least one vertex $v \in V(H)$, that has at most one neighbour
  $u \in V(G)\backslash V(H)$. The set $V(H) \backslash \{v\}$ forms a
  positive zero forcing set. To see this, first note that after
  removing this set, all the vertices in $V(G) \backslash ( V(H) \cup
  \{u\})$ are disjoint. Finally, there is one vertex in $V(H)$ that is
  not adjacent to $u$, this vertex can force $v$.

  For the third statement assume $k$ is even and $|S(G)|=k+1$. Since
  $K_{k+1}$ is a subgraph of $G$, the tree cover number of $G$ is no
  less than $\lceil \frac{k+1}{2} \rceil$. Suppose that $\TT$ is a
  minimal tree covering for $G$ with $|\TT|=\lceil \frac{k+1}{2}
  \rceil$. Since no tree in $\TT$ can contain more than two vertices
  of $H$ and $T(G) = \lceil \frac{k+1}{2} \rceil$, each tree in
  $\TT$, except one, contains exactly two vertices of $H$. Assume that
  $T_1$ is the tree that contains only a single vertex of $H$.

  The size of $S(G)$ is $k+1$, so for any vertex, $w\in V(H)$, there
  is a corresponding vertex in $V(G)\backslash V(H)$ which is adjacent
  to all of the vertices of $H$ except $w$. In particular, if $w$ is
  the single vertex of $H$ in the tree $T_1$, then there is a vertex
  $u$ which is adjacent to all vertices in $H$ except $w$.  Since
  $N_G(u)=V(H)\backslash \{w\}$, the vertex $u$ can not be covered by
  extending any of the trees of $\TT$ (as $u$ is adjacent to both
  vertices in any tree from $\TT$). This contradicts $\TT$ being a
  tree covering for $G$. Thus
\[
T(G)\geq\left\lceil \frac{k+1}{2}\right\rceil+1.
\] 

Finally we will show that we can construct a tree covering of this
size.  Let $\TT'$ be a minimal tree covering for $H$. Thus
$|\TT'|=\lceil \frac{k+1}{2} \rceil$ and exactly one tree in $\TT'$
contains only one vertex (all other trees contain exactly two
vertices).  Call this tree $T=\{u\}$. Extend $T$ to include every
vertex in $V(G)\backslash V(H)$ that is adjacent to $u$. Now the only
vertices in $G$ that are not covered by a tree in $\TT'$ are the
vertices in $V(G)\backslash V(H)$ that are adjacent to every vertex in
$H$, except $u$, call these vertices $v_1,v_2,\dots, v_\ell$.  Take
any tree in $\TT'$, except $T =\{u\}$. Then this tree will have two vertices,
say $\{x,y\}$.  Remove this tree from $\TT'$ and replace it with the
two trees, $\{x\}$ and $\{ \{y,v_1\},\{y,v_1\}, \dots,\{y,v_\ell\}\}$. 
This gives a tree covering of size $\lceil \frac{k+1}{2} \rceil + 1$.

%
%
%

For the fourth statement, assume $|S(G)|<k+1$ and $k$ is even. Thus
there is a vertex, $v \in V(H)$ that is adjacent to all of the
vertices in the graph $G$. Let $\TT'$ be a tree covering of $H$ with
$\lceil \frac{k+1}{2} \rceil$ trees in which $T=\{v\}$ is a covering tree.
Then $T$ can be extended to cover all the vertices of $V(G)\backslash
V(H)$. Thus $T(G)=|\TT'|=\lceil \frac{k+1}{2} \rceil$.
\end{proof}

Note that if in Theorem~\ref{Z_+&T of k-trees} we have $k=2$, then the
positive zero forcing number and the tree cover number coincide (this
result is also proved in \cite{ekstrand2011note}).
 
\begin{thm}\label{k-tree with k odd}
Let $G$ be a $k$-tree with $k$ is odd, then $T(G)=\frac {k+1}{2}$. 
\end{thm}
\begin{proof}
  We prove this by induction on the number of vertices in $G$. The result is clearly
  true for $K_4$ (the smallest $3$-tree). Assume that the statement is
  true for all $G'$ with $|V(G')|<n$. Let $G$ be a $k$-tree with
  $|V(G)|=n$.  

  Let $v\in V(G)$ be a vertex of degree $k$ (such a vertex always
  exists).  By the induction hypothesis there exists a tree covering
  $\TT$ of $G\backslash{v}$ with exactly $\frac {k+1}{2}$ trees. Since
  neighbours of $v$ form a $k$-clique, any tree in $\TT$ covers at
  most two of the neighbours of $v$. Moreover, $v$ has an odd number
  of neighbours, thus there is exactly one tree $T$ in $\TT$ that
  covers only one of the neighbours of $v$. Therefore we can extend
  $T$ to cover $v$, and conclude $T(G)=\frac {k+1}{2}$.
\end{proof}

Recall that a graph is called {\em chordal} if it contains no induced cycles on four or more
vertices. For instance, all $k$-trees are examples of chordal graphs. In general it is known
for any chordal graph $G$, that $\Mp(G) = |G| - cc(G)$, where $cc(G)$ denotes the fewest number of cliques
needed to cover (or to include) all the edges in $G$ (see \cite{MN}). This number, $cc(G)$, is often called the 
{\em clique cover number} of the graph $G$. Further inspection of the work in 
\cite{MN} actually reveals that, in fact, for any chordal graph, $cc(G)$ is equal to the {\em ordered set number}
($OS(G)$) of $G$. In \cite{MR2645093}, it was proved that for any graph $G$, the ordered set number of $G$ and
the positive zero forcing number of $G$ are related and satisfy, $Z_{+}(G)+ OS(G) = |G|$. As a consequence,
we have that $\Mp(G)=Z_{+}(G)$ for any chordal graph $G$, and, in particular, $Z_{+}(G) = |G|-cc(G)$. So 
studies of the positive zero forcing number of chordal graphs, including $k$-trees, boils down to determining the clique cover number and vice-versa. 
 

\section{Further work}\label{conclusions}

In Section~\ref{Graphs with Z(G)=P(G)} we introduced families of
graphs for which the zero forcing number and the path cover number
coincide. In fact, we showed that for the family of block-cycle graphs
this is true. However,  there are additional families for
which equality holds between these two parameters. For example,
the graph $G=K_4-e$, where $e$ is an edge of $K_4$, has
$Z(G)=P(G)$. It is, therefore, natural to propose characterizing
all the graphs $G$ for which $Z(G)=P(G)$. 

In~\cite{Fatemeh} it is conjectured the result analogous to
Corollary~\ref{Z+vertexsum} holds for zero forcing sets; if this
conjecture is confirmed, then there would be a much large family of
graphs for which the path number and zero forcing number coincide. To this 
end, we state the following problem as a beginning to this study.

\begin{conj}\label{conj_Z&P_vertex_sum}
Let $G$ and $H$ be two graphs, both with an identified vertex $v$, and both 
satisfy $Z(G)=P(G)$ and $Z(H)=P(H)$. Then
\[
 Z(G\,\,\stackplus{v}\,H)=P(G\,\stackplus{v}\,H).
\]
\end{conj} 

It is not difficult to verify that in any tree, any minimal path cover
coincides with a collection of forcing chains. We conjecture that this
is also the case for the block-cycle graphs (and refer the reader to
\cite[Section 5.2]{Fatemeh} for more details). In general, it is an
interesting question if for a graph $G$ with $Z(G)=P(G)$, is it true
that any minimal path cover of $G$ coincides with a collection of
forcing chains of $G$?

In Section~\ref{forcing_trees}, we proved the equality $Z_{+}(G) =
T(G)$ where $G$ is an outerplanar graph. The structure of a planar
embedding of outerplanar graphs was the key point to establishing the
equality. There are many non-outerplanar graphs with a similar
structure; generalizing this structure will lead to discovering more
graphs that satisfy $Z_+(G)=T(G)$. In general, we are interested in
characterizing all the graphs $G$ for which $Z_{+}(G) = T(G)$.




\begin{thebibliography}{10}

\bibitem{aim} AIM Minimum Rank -- Special Graphs Work Group.  Zero forcing sets and the minimum rank  of graphs.   {\em Linear Algebra  Appl.}, 428/7: 1628--1648, 2008.


\bibitem{Fatemeh}
F. Alinaghipour Taklimi.
\newblock {\em Zero Forcing Sets for Graphs}
\newblock University of Regina, 2013.
\newblock Ph.D. thesis.


\bibitem{estrella} 

E. Almodovar, L. DeLoss, L. Hogben, K. Hogenson, K. Myrphy, T. Peters, C. Ram\'rez, C.  
\newblock Minimum  rank, maximum nullity and zero forcing number, and spreads of these parameters for selected graph families. 
\newblock {\em Involve.  A Journal of Mathematics}, 3: 371--392, 2010.

\bibitem{MR2645093}
F. Barioli, W. Barrett, S. Fallat, H.~T. Hall, L. Hogben, B. Shader, P.~van~den Driessche, and H. van~der Holst.
\newblock Zero forcing parameters and minimum rank problems.
\newblock {\em Linear Algebra Appl.}, 433(2): 401--411, 2010.

\bibitem{param}	
F. Barioli, W. Barrett, S. Fallat, H. T. Hall,  L. Hogben, B. Shader, P. van den Driessche, H. van der Holst.  
\newblock Parameters related to tree-width, zero forcing, and maximum  nullity of a graph.  
\newblock {\em Journal of Graph Theory}, 72: 146-177, 2013.


\bibitem{BFH}
F. Barioli, S. Fallat, and L. Hogben.
\newblock On the difference between the maximum multiplicity and path cover number for
tree-like graphs.
\newblock {\em Linear Algebra Appl.}, 409: 13--31, 2005.

\bibitem{barioli2011minimum}
F. Barioli, S.~M. Fallat, L.~H. Mitchell, and S.~K. Narayan.
\newblock Minimum semidefinite rank of outerplanar graphs and the tree cover number.
\newblock {\em Electronic Journal of Linear Algebra}, 22: 10--21, 2011.

\bibitem{burgarth2007full}
D. Burgarth and V. Giovannetti.
\newblock Full control by locally induced relaxation.
\newblock {\em Physical Review Letters}, 99(10): 100--501, 2007.

\bibitem{burgarth2009indirect}
D. Burgarth and K. Maruyama.
\newblock Indirect {H}amiltonian identification through a small gateway.
\newblock {\em New Journal of Physics}, 11(10): 103019, 2009.

\bibitem{ZFQC} 
D. Burgarth, D. D'Alessandro, L. Hogben, S. Severini, M. Young.
\newblock Zero forcing, linear and quantum controllability for systems evolving on networks. 
\newblock {\em IEEE Transactions in Automatic Control}, 58(9): 2349--2354, 2013.

\bibitem{Die}
R. Diestal. {\em Graph Theory}. Graduate Texts in Mathematics, 137, Springer-Verlag, Heidelberg, 2000.

\bibitem{EHHLR} 
C. J. Edholm, L. Hogben, M. Huynh, J. Lagrange, D. D. Row. 
\newblock Vertex and edge spread of zero forcing number, maximum nullity, and  minimum rank of a graph. 
\newblock {\em Linear Algebra and its Applications}, 436: 4352--4372, 2012.

\bibitem{ekstrand2011positive}
J. Ekstrand, C. Erickson, H.~T. Hall, D. Hay, L. Hogben, R. Johnson, N. Kingsley, S. Osborne, T. Peters, J. Roat, et~al.
\newblock Positive semidefinite zero forcing.
\newblock {\em Linear Algebra and its  Applications}, 439: 1862--1874, 2013.

\bibitem{ekstrand2011note}
J. Ekstrand, C. Erickson, D. Hay, L. Hogben, and J. Roat.
\newblock Note on positive semidefinite maximum nullity and positive semidefinite zero forcing number of partial $2$-trees.
\newblock  {\it Electronic Journal of Linear Algebra},  23: 79--97,  2012. 

\bibitem{MN}
P. Hackney, B. Harris, M. Lay, L. H. Mitchell, S. K. Narayan, and A. Pascoe.
Linearly independent vertices and minimum semidefinite rank.
{\em Linear Algebra and its Applications}, 431: 1105 - 1115, 2009.


\bibitem{Yeh}
L.-H. Huang, G. J. Chang, H.-G. Yeh.  
\newblock On minimum rank and zero forcing sets of a graph. 
\newblock {\em Linear Algebra and its Applications}, 432: 2961--2973, 2010.

\bibitem{MR2388646}
AIM Minimum Rank-Special Graphs~Work Group.
\newblock Zero forcing sets and the minimum rank of graphs.
\newblock {\em Linear Algebra Appl.}, 428(7): 1628--1648, 2008.

\bibitem{johnson2009graphs}
C.~R. Johnson, R. Loewy, and P.~A. Smith.
\newblock The graphs for which the maximum multiplicity of an eigenvalue is two.
\newblock {\em Linear and Multilinear Algebra}, 57(7): 713--736, 2009.

\bibitem{MR2917419}
D. Row.
\newblock{A technique for computing the zero forcing number of a graph with a cut-vertex.}
\newblock {\em Linear Algebra Appl.}, 436(12): 4423--4432, 2012.	

\bibitem{Sev}
S. Severini. 
\newblock Nondiscriminatory propagation on trees.
\newblock{\em Journal of Physics A}, 41: 482--002 (Fast Track Communication), 2008.	

\bibitem{Yang}
B. Yang.
\newblock Fast-mixed searching and related problems on graphs.
\newblock {\em Theoretical Computer Science}, 507(7): 100--113, 2013.

\end{thebibliography}
\end{document}